\documentclass[12pt,leqno]{article}
\usepackage{amsfonts}
\usepackage{amsmath, amsthm, amsfonts, amssymb, color}
\usepackage{mathrsfs}
\usepackage[notcite,notref]{showkeys}

\newtheorem{thm}{Theorem}[section]

\newtheorem{lem}[thm]{Lemma}

\newtheorem{exa}[thm]{Example}
\theoremstyle{definition}

\newcommand{\scr}[1]{\mathscr #1}
\definecolor{wco}{rgb}{0.5,0.2,0.3}

\numberwithin{equation}{section} \theoremstyle{remark}
\newtheorem{rem}{Remark}[section]

\newcommand{\ua}{\uparrow}

\title{{\bf Transportation Cost Inequalities for Neutral Functional SDEs}
}
\author{
{\bf  Jianhai Bao and Chenggui Yuan}\\
 \footnotesize{Department of Mathematics,
Swansea University, Singleton Park, SA2 8PP, UK}\\
\footnotesize{ majb@Swansea.ac.uk, \, C.Yuan@swansea.ac.uk}}
\begin{document}
\def\R{\mathbb R}  \def\ff{\frac} \def\ss{\sqrt} \def\B{\mathbf
B}
\def\N{\mathbb N} \def\kk{\kappa} \def\m{{\bf m}}
\def\dd{\delta} \def\DD{\Delta} \def\vv{\varepsilon} \def\rr{\rho}
\def\<{\langle} \def\>{\rangle} \def\GG{\Gamma} \def\gg{\gamma}
  \def\nn{\nabla} \def\pp{\partial} \def\EE{\scr E}
\def\d{\text{\rm{d}}} \def\bb{\beta} \def\aa{\alpha} \def\D{\scr D}
  \def\si{\sigma} \def\ess{\text{\rm{ess}}}
\def\beg{\begin} \def\beq{\begin{equation}}  \def\F{\scr F}
\def\Ric{\text{\rm{Ric}}} \def\Hess{\text{\rm{Hess}}}
\def\e{\text{\rm{e}}} \def\ua{\underline a} \def\OO{\Omega}  \def\oo{\omega}
 \def\tt{\tilde} \def\Ric{\text{\rm{Ric}}}
\def\cut{\text{\rm{cut}}} \def\P{\mathbb P} \def\ifn{I_n(f^{\bigotimes n})}
\def\C{\scr C}      \def\aaa{\mathbf{r}}     \def\r{r}
\def\gap{\text{\rm{gap}}} \def\prr{\pi_{{\bf m},\varrho}}  \def\r{\mathbf r}
\def\Z{\mathbb Z} \def\vrr{\varrho} \def\l{\lambda}
\def\L{\scr L}\def\Tt{\tt} \def\TT{\tt}\def\II{\mathbb I}
\def\i{{\rm in}}\def\Sect{{\rm Sect}}\def\E{\mathbb E} \def\H{\mathbb H}
\def\M{\scr M}\def\Q{\mathbb Q} \def\texto{\text{o}} \def\LL{\Lambda}
\def\Rank{{\rm Rank}} \def\B{\scr B} \def\i{{\rm i}} \def\HR{\hat{\R}^d}
\def\to{\rightarrow}\def\l{\ell}
\def\8{\infty}

\maketitle

\begin{abstract}
A class of functional differential equations are investigated. Using the Girsanov-transformation argument
 we establish the quadratic transportation cost inequalities  for a class of
finite-dimensional  neutral functional stochastic differential
equations and infinite-dimensional neutral functional
stochastic partial differential equations under different metrics.
\\


 \end{abstract}
 \vskip 2cm

\section{Introduction}

Let $(E,\d)$ be a  metric space equipped
with a $\sigma$-algebra $\mathscr{B}(E)$ such that metric
$\d(\cdot,\cdot)$ is  $\mathscr{B}(E)\times  \mathscr{B}(E)$ measurable. For any $p\geq1$ and  probability measures $\mu$
and $\nu$ on $(E,\mathscr{B}(E),\d)$, the $L^p$-Wasserstein distance
between $\mu$ and $\nu$ is defined by
\begin{equation*}
W_{p,\d}(\mu,\nu):=\inf_{\pi\in\mathscr{C}(\mu,\nu)}\Big\{\int_{E\times
E}\d^p(x,y)\pi(\d x,\d y)\Big\}^{1/p},
\end{equation*}
where $\mathscr{C}(\mu,\nu)$ denotes the totality of probability
measures on $E\times E$ with the marginals $\mu$ and $\nu$. In many
practical situations, it is quite useful to find an upper bound for
the metric $W_{p,\d}(\mu,\nu)$, where a fully satisfactory one given
by  Talagrand  \cite{T96} is the relative entropy of $\nu$ with
respect to $\mu$
\begin{equation*}
\mbox{H}(\nu|\mu):=
\begin{cases}
\int\ln \frac{\d \nu}{\d \mu}\d \nu, \ \ \ \nu\ll \mu,\\
+\infty, \ \ \ \ \ \ \ \ \mbox{ otherwise},
\end{cases}
\end{equation*}
where $\nu\ll\mu$ means that $\nu$ is absolutely continuous with
respect to $\mu$. We  say that the probability measure $\mu$ on
$(E, \d)$ satisfies an $L^p$ transportation cost
inequality (TCI) with a constant $C>0$ if
\begin{equation*}
W^2_{p,d}(\mu,\nu)\leq2 C\mbox{H}(\nu|\mu)
\end{equation*}
for any probability measure $\nu$. As
usual, we write $\mu\in T_p(C)$ for this relation.

Concentration inequalities and their applications have become an
integral part of modern probability theory. Here we highlight three
 monographs: Ledoux \cite{L}, \"Ust\"unel \cite{U} and  Villani
\cite{V03}. One of the powerful tools to show such concentration
estimates for diffusions is TCI, where quadratic TCI is unique in
its advantages and related to the log-Sobolev inequality,
hypercontractivity, Poincar\'{e} inequality, inf-convolution, and
Hamilton-Jacobi equations, for details, see, e.g., Bobkov and
G\"otze \cite{BG99}, Gozlan, Roberto and Samson \cite{GRS11} Otto
and Villani \cite{OV}. Recently, there are extensive literature on
the topic of transportation cost inequalities (TCIs) in all kinds of
path spaces, e.g., in \cite{DGW04} Djellout, Guilin and Wu
investigate the stability under weak convergence, $T_1(C), T_2(C)$
and applications for random dynamical systems on Wiener space,  in
\cite{P11} Pal shows that probability laws of certain
multidimensional semimartingales satisfy quadratic TCI,  in
\cite{U10}  \"Ust\"unel proves TCI for the laws of diffusion
processes where the drift  depends on the full history.
 On Poisson space Wu \cite{W10} develops the $W_1H$ inequality for stochastic differential equation (SDEs)
with pure jumps, and Ma \cite{M10} discusses the $W_1H$ inequality
for SDEs driven by the Brownian motion  and the jump  together.
Wang  \cite{W04} establishes some TCIs
on the path space over a connected complete Riemannian manifold with
Ricci curvature bounded from below.   Fang and  Shao \cite{FS07}
provides  optimal transport maps for Monge-Kantorovich problem on
connected Lie groups, to name a few.

Moreover, many other methods have  been applied to establish TCIs,
e.g., Large deviation principle is used in \cite{GL07},   and
ensorization is implemented in \cite{W04}. In particular, the
Girsanov transformation theorem has been utilized in \cite{DGW04} to
provide a characterization of $L^1$-TCIs for diffusions, which has
been generalized to different setting, e.g., infinite-dimensional
dynamical systems in \cite{WZ06}, time-inhomogenous diffusions in
\cite{P11}, multivalued SDEs and singular SDEs in \cite{U10}, and
SDEs driven by a fractional Brownian motion in \cite{S11}. On the
other hand, due to the Girsanov transformation is unavailable for
the jump cases, recently Malliavin calculus technique is applied in
\cite{M10, W10}.

Motivated by the previous works,
 and references therein,  we shall establish
the $T_2$-TCIs for a class of finite-dimensional neutral functional
SDEs and infinite-dimensional  neutral functional stochastic partial
differential equations (SPDEs). A neutral functional differential
equation (see \cite{HH98,L06, M97}) is one in which the derivatives
of the past history are involved, as well as those of the present
state of the system.

The rest of the paper is organised as follows:  the $T_2$-TCIs for a
class of finite-dimensional  neutral functional SDEs are studied in
Section 2, in Section 3 we show the $T_2$-TCIs for
infinite-dimensional neutral functional stochastic partial
differential equations (SPDEs).

\section{TCIs for Neutral Functional SDEs}
For a positive integer $n$,  let
$(\mathbb{R}^n,\langle\cdot,\cdot\rangle,|\cdot|)$ be the Euclidean
space and $\|A\|_{HS}:=\sqrt{\mbox{trace}(A^*A)}$, the
Hilbert-Schmidt norm for a matrix $A$, where $A^*$ is the  transpose
of $A$. Let  $\tau>0$  and $T>0$ be two constants.  Denote
$\mathscr{C}:=\mathcal {C}([-\tau,0];\mathbb{R}^n)$ by the family of
continuous functions $\xi:[-\tau,0]\mapsto\mathbb{R}^n$ and  $
\mathcal {C}([-\tau,T];\mathbb{R}^n)$ the set of continuous
functions on $[-\tau,T]$.  Let $X \in \mathcal
{C}([-\tau,T];\mathbb{R}^n)$ and $t\in[0,T]$, we define the segment
$X_t\in\mathscr{C}$ of $X$ by $X_t(\theta):=X(t+\theta)$ for
$\theta\in[-\tau,0]$, where $X(t)\in\mathbb{R}^n$ is a point, while
$X_t\in\mathscr{C}$ is a continuous $\mathbb{R}^n$-valued function
on $[-\tau,0]$. Denote $\mathcal {W}([-\tau,0];\mathbb{R}_+)$ by the
family of Borel-measurable functions
$w:[-\tau,0]\mapsto\mathbb{R}_+$ such that
$\int_{-\tau}^0w(\theta)\d \theta=1$. Throughout the paper $C>0$ is
a generic constant whose values may change for its different
appearances.

Consider neutral functional SDE on $\mathbb{R}^n$
\begin{equation}\label{eq1}
\begin{cases}
\d[X(t)-G(X_t)]=b(X_t)\d t+\sigma(X_t)\d W(t), \ \ \ t\in[0,T],\\
X_0=\xi\in\mathscr{C}.
\end{cases}
\end{equation}
Here $G,b:\mathscr{C}\rightarrow\mathbb{R}^n$,
$\sigma:\mathscr{C}\rightarrow\mathbb{R}^{n\times m}$, and $W$ is an
$m$-dimensional Brownian motion defined on some filtered complete
probability space $(\Omega,\mathcal {F},\{\mathcal
{F}_{t}\}_{t\geq0},\mathbb{P})$. Assume that $G(0)=0$ and there
exist $\kappa\in(0,1)$ and $w_1\in\mathcal
{W}([-\tau,0];\mathbb{R}_+)$ such that
\begin{equation}\label{eq7}
|G(\xi)-G(\eta)|^2\leq\kappa\int_{-\tau}^0w_1(\theta)|\xi(\theta)-\eta(\theta)|^2\d
\theta, \ \ \ \xi,\eta\in\mathscr{C}.
\end{equation}
We further assume that $b$ and $\sigma$ are locally Lipschitzian and
there exist $\delta>0,\lambda_1\in\mathbb{R}, \lambda_2\geq0$ and
$w_2\in\mathcal {W}([-\tau,0];\mathbb{R}_+)$ such that
\begin{equation}\label{eq9}
2\langle\xi(0)-G(\xi),b(\xi)\rangle+\|\sigma(\xi)\|_{HS}^2
\leq\delta\Big(1+|\xi(0)|^2+\int_{-\tau}^0|\xi(\theta)|^2\d\theta\Big),\
\ \ \xi\in\mathscr{C},
\end{equation}
and for arbitrary $\xi,\eta\in\mathscr{C}$
\begin{equation}\label{eq8}
\begin{split}
2\langle\xi(0)&-G(\xi)-(\eta(0)-G(\eta)),b(\xi)-b(\eta)\rangle+\|\sigma(\xi)-\sigma(\eta)\|_{HS}^2\\
&\leq-\lambda_1|\xi(0)-\eta(0)|^2+\lambda_2\int_{-\tau}^0w_2(\theta)|\xi(\theta)-\eta(\theta)|^2\d\theta.
\end{split}
\end{equation}

\begin{rem}
{\rm If we take $G(\xi)=G(\xi(-\tau)),
b(\xi)=b(\xi(0),\xi(-\tau)),\si(\xi)=\si(\xi(0),\xi(-\tau)),\xi\in\mathscr{C},$
then the equation \eqref{eq1} becomes a delay SDE, that is:
\begin{equation*}
\begin{cases}
\d[X(t)-G(X(t-\tau)]=b(X(t),X(t-\tau))\d t+\sigma(X(t),X(t-\tau))\d W(t), \ \ \ t\in[0,T],\\
X_0=\xi\in\mathscr{C}.
\end{cases}
\end{equation*}
There are many examples satisfying \eqref{eq7}-\eqref{eq8}. Actually
\eqref{eq7} becomes 
$|G(\xi(-\tau))-G(\eta(-\tau))|$ $\leq\kappa|\xi(-\tau)-\eta(-\tau)|$
for some $\kappa\in(0,1)$ by taking $w(\theta)=1/\tau$ for
$\theta\in[-\tau,0]$ (In fact, for the delay we can take
$\kappa\geq1$ by an induction argument). Moreover, since $b$ and
$\sigma$ are locally Lipschitzian, Eq. \eqref{eq1} has a unique
local solution. On the other hand, \eqref{eq9} can suppress the
explosion of the solution in a finite-time interval. Therefore Eq.
\eqref{eq1} has a unique solution
$\{X(t,\xi)\}$ on $[0, T]$. }
\end{rem}

Let  $T=N$ be  a positive integer. For any $\gamma_1, \gamma_2\in
\mathcal {X}:=\mathcal {C}([0,T];\mathbb{R}^n)$, we define
\begin{equation*}
\d_{\infty,1}(\gamma_1,\gamma_2):=\Big\{\sum_{k=0}^{N-1}\sup_{k\leq
t\leq k+1}|\gamma_1(t)-\gamma_2(t)|^2\Big\}^{1/2}, \ \ \
\gamma_1,\gamma_2\in\mathcal {X}.
\end{equation*}
Then  $\mathcal {X}$ is a Banach space  with respect to
$\d_{\infty,1}$.

\begin{thm}\label{Theorem 1}
{\rm Let \eqref{eq7}-\eqref{eq8} hold and $\mathbb{P}_\xi$ be the
law of $X(\cdot,\xi)$, solution process of Eq. \eqref{eq1}. Assume
that $\sigma$ is bounded by
$\tilde{\sigma}:=\sup_{\xi\in\mathscr{C}}\|\sigma(\xi)\|_{HS}<\infty$,
$\lambda_1-\lambda_2>0$, and there exists $\lambda_3>0$ such that
\begin{equation}\label{eq38}
\|\sigma(\xi)-\sigma(\eta)\|_{HS}^2\leq\lambda_3\int_{-\tau}^0|\xi(\theta)-\eta(\theta)|^2\d\theta,
\ \ \ \xi,\eta\in\mathscr{C}.
\end{equation}
Then  $\mathbb{P}_\xi\in T_2(C)$ on $\mathcal {X}$ under the metric
$\d_{\infty,1}$, where $C>0$ is dependent on
$\kappa,\lambda_i,i=1,2,3,\tilde{\sigma}, N$ and the universal
constant from the Burkhold-Davis-Gundy inequality. Moreover, under
the previous assumptions, except \eqref{eq38},  $\mathbb{P}_\xi\in
T_2(C)$ for some $C>0$ under the metric
\begin{equation*}
\d_{L^2}(\gamma_1,\gamma_2):=\Big(\int_0^T|\gamma_1(t)-\gamma_1(t)|^2\d
t\Big)^{\frac{1}{2}},\ \ \ \gamma_1,\gamma_2\in\mathcal {X},
\end{equation*}
where $C>0$ is independent of $T$ whenever $\lambda_1-\lambda_2>0$,
otherwise dependent on $T$. }
\end{thm}

\begin{proof}
It should be  pointed out that the argument is motivated by that of
\cite{DGW04,WZ04}, where the key point is to express the finiteness
of the entropy by means of the energy of the drift from the Girsanov
transformation of a well chosen probability measure. In the sequel
we divide the proof into two steps to show the desired assertions under two
different metrics.

({\bf1}) Let $\mathbb{P}_\xi$ be the law of $X(\cdot,\xi)$ on
$\mathcal {X}$ and $\mathbb{Q}$ be any probability measure on $\mathcal
{X}$ such that $\mathbb{Q}\ll \mathbb{P}_\xi$. Define
\begin{equation}\label{eq17}
\tilde{\mathbb{Q}}:=\frac{\d \mathbb{Q}}{\d
\mathbb{P}_\xi}(X(\cdot,\xi))\mathbb{P},
\end{equation}
which is a probability measure on $(\Omega,\mathcal {F})$. Recalling
the definition of ``entropy'' and adopting a measure-transformation
argument we obtain from \eqref{eq17} that
\begin{equation*}
\begin{split}
\mbox{\bf H}(\tilde{\mathbb{Q}}|\mathbb{P})&=\int_\Omega\ln
\Big(\frac{\d \tilde{\mathbb{Q}}}{\d \mathbb{P}}\Big)\d
\tilde{\mathbb{Q}}=\int_\Omega\ln \Big(\frac{\d \mathbb{Q}}{\d
\mathbb{P}_\xi}(X(\cdot,\xi))\Big)\frac{\d \mathbb{Q}}{\d
\mathbb{P}_\xi}(X(\cdot,\xi))\d \mathbb{P}\\
&=\int_{\mathcal {X}}\ln \Big(\frac{\d \mathbb{Q}}{\d
\mathbb{P}_\xi}\Big)\frac{\d \mathbb{Q}}{\d \mathbb{P}_\xi}\d
\mathbb{P}_\xi\\
&=\mbox{\bf H}(\mathbb{Q}|\mathbb{P}_\xi).
\end{split}
\end{equation*}
Additionally, by the martingale representation theorem, there exists
by, e.g., \cite{DGW04,WZ04}, a predictable process
$h\in\mathbb{R}^m$ with $\int_0^t|h(s)|^2\d s<\infty$ for
$t\in[0,N]$, $\mathbb{P}$-\mbox{a.s.}, such that
\begin{equation}\label{eq4}
\mbox{\bf H}(\tilde{\mathbb{Q}}|\mathbb{P})=\mbox{\bf
H}(\mathbb{Q}|\mathbb{P}_\xi)=\frac{1}{2}\mathbb{E}^{\tilde{\mathbb{Q}}}\int_0^N|h(t)|^2\d t.
\end{equation}
Furthermore, due to the Girsanov theorem,
\begin{equation*}
\tilde{W}(t):=W(t)-\int_0^th(s)\d s,\ \ \ t\in[0,N],
\end{equation*}
is a Brownian motion with respect to $\{\mathcal {F}_t\}_{t\geq0}$
on the probability space $(\Omega,\mathcal {F},\tilde{\mathbb{Q}})$.
Then,  under the measure $\tilde{\mathbb{Q}}$, the process
$\{X(t,\xi)\}_{t\in[0,N]}$ satisfies
\begin{equation}\label{eq2}
\begin{cases}
\d[X(t)-G(X_t)]=[b(X_t)+\sigma(X_t)h(t)]\d t+\sigma(X_t)\d \tilde{W}(t),\ \ \ t\in[0,N],\\
X_0=\xi.
\end{cases}
\end{equation}
Observe that, to show $\mathbb{P}_\xi\in T_2(C)$ for some $C>0$, we
need to bound the Wasserstein distance between $\mathbb{P}_\xi$ and
$\mathbb{Q}$, and couple the solutions of Eq. \eqref{eq1} and Eq.
\eqref{eq2}. Let $\{Y(t,\xi)\}_{t\in[0,N]}$ be the solution of the
following equation
\begin{equation}\label{eq3}
\begin{cases}
\d[Y(t)-G(Y_t)]=b(Y_t)\d t+\sigma(Y_t)\d \tilde{W}(t),\ \ \ t\in[0,N],\\
Y_0=\xi.
\end{cases}
\end{equation}
By virtue of the uniqueness, under $\tilde{\mathbb{Q}}$ the law of
$Y(\cdot,\xi)$ is $\mathbb{P}_\xi$. Consequently $(X,Y)$ under
$\tilde{\mathbb{Q}}$ is a coupling of $(\mathbb{Q},\mathbb{P}_\xi)$
and, by the definition of ``$L^2$-Wasserstein distance'', we have
\begin{equation}\label{eq18}
(W_{2,\d}(\mathbb{Q},\mathbb{P}_\xi))^2\leq\mathbb{E}^{\tilde{\mathbb{Q}}}(\d_{\infty,1}^2(X,Y)).
\end{equation}
Thus, by \eqref{eq4} and \eqref{eq18},  to obtain the desired
Talagrand's inequality, it suffices to estimate the
$\d_{\infty,1}$-distance between $X$ and $Y$, which is bounded by
$\mathbb{E}^{\tilde{\mathbb{Q}}}\int_0^N|h(t)|^2\d t$ up to some
constant. That is, we only need to verify that for $t\in[0,N]$
\begin{equation}\label{eq13}
\mathbb{E}^{\tilde{\mathbb{Q}}}(\d_{\infty,1}^2(X,Y))=\sum_{k=0}^{N-1}\mathbb{E}^{\tilde{\mathbb{Q}}}\Big(\sup_{k\leq
t\leq k+1}|X(t)-Y(t)|^2\Big)\leq
C\mathbb{E}^{\tilde{\mathbb{Q}}}\int_0^N|h(t)|^2\d t.
\end{equation}
Recall a fundamental inequality: for any $a,b>0$ and
$\epsilon\in(0,1)$,
\begin{equation}\label{eq6}
(a+b)^2\leq a^2/\epsilon+b^2/(1-\epsilon).
\end{equation}
This, together with \eqref{eq7}, yields that
\begin{equation*}
\begin{split}
\mathbb{E}^{\tilde{\mathbb{Q}}}|X(t)-Y(t)|^2&\leq\frac{1}{1-\sqrt{\kappa}}\mathbb{E}^{\tilde{\mathbb{Q}}}|M(t)|^2+\sqrt{\kappa}\int_{-\tau}^0w_1(\theta)\mathbb{E}^{\tilde{\mathbb{Q}}}|X(t+\theta)-Y(t+\theta)|^2\d\theta,
\end{split}
\end{equation*}
where $M(t):=X(t)-G(X_t)-(Y(t)-G(Y_t)),t\in[0,N]$. Due to
$w_1\in\mathcal {W}([-\tau,0];\mathbb{R}_+)$ and
$X(\theta)=Y(\theta)$ for $\theta\in[-\tau,0]$, it is easy to note
that
\begin{equation}\label{eq12}
\begin{split}
\sup_{0\leq s\leq
t}\mathbb{E}^{\tilde{\mathbb{Q}}}|X(s)-Y(s)|^2&\leq\frac{1}{(1-\sqrt{\kappa})^2}\sup_{0\leq
s\leq t}\mathbb{E}^{\tilde{\mathbb{Q}}}|M(s)|^2,\ \ \ t\in[0,N].
\end{split}
\end{equation}
Furthermore,  by the It\^o formula and the Young inequality,
together with \eqref{eq7}, \eqref{eq8}, the boundedness of $\sigma$
and  $w_1,w_2\in\mathcal {W}([-\tau,0];\mathbb{R}_+)$, we obtain
that for  arbitrary $\epsilon\in(0,1)$
\begin{equation}\label{eq34}
\begin{split}
&\mathbb{E}^{\tilde{\mathbb{Q}}}|M(t)|^2\\&\leq-\lambda_1\mathbb{E}^{\tilde{\mathbb{Q}}}\int_0^t|X(s)-Y(s)|^2\d
s+\lambda_2\mathbb{E}^{\tilde{\mathbb{Q}}}\int_0^t\int_{-\tau}^0w_2(\theta)|X(s+\theta)-Y(s+\theta)|^2\d
\theta\d
s\\
&\quad+2\tilde{\sigma}\mathbb{E}^{\tilde{\mathbb{Q}}}\int_0^t|M(s)||h(s)|\d s\\
&\leq-\lambda_1\mathbb{E}^{\tilde{\mathbb{Q}}}\int_0^t|X(s)-Y(s)|^2\d
s+\lambda_2\mathbb{E}^{\tilde{\mathbb{Q}}}\int_{-\tau}^0w_2(\theta)\int_0^t|X(s)-Y(s)|^2\d
s\d
\theta\\
&\quad+\mathbb{E}^{\tilde{\mathbb{Q}}}\int_0^t\Big\{\frac{\epsilon}{1+\kappa}\Big(|X(s)-Y(s)|^2+\kappa\int_{-\tau}^0w_1(\theta)|X(s+\theta)-Y(s+\theta)|^2\d
\theta\Big)\\
&\quad+\frac{2(1+\kappa)\tilde{\sigma}^2}{\epsilon}|h(s)|^2\Big\}\d s\\
&\leq-\tilde{\lambda}\mathbb{E}^{\tilde{\mathbb{Q}}}\int_0^t|X(s)-Y(s)|^2\d
s+\frac{2(1+\kappa)\tilde{\sigma}^2}{\epsilon}\int_0^t\mathbb{E}^{\tilde{\mathbb{Q}}}|h(s)|^2\d
s, \ \ \ t\in[0,N],
\end{split}
\end{equation}
where $\tilde{\lambda}:=\lambda_1-\lambda_2-\epsilon$, and   we have
also used the fact that the stochastic integral has zero expectation
under $\tilde{\mathbb{Q}}$ by a standard stopping time trick. Choose $\epsilon\in(0,1)$ sufficiently small such that
$\tilde{\lambda}>0$  and note that
$t\in[0,N]$,
\begin{equation}\label{eq45}
\sup\limits_{0\leq s\leq
t}\mathbb{E}^{\tilde{\mathbb{Q}}}|M(s)|+\tilde{\lambda}\mathbb{E}^{\tilde{\mathbb{Q}}}\int_0^t|X(s)-Y(s)|^2\d
s
\le 2\sup\limits_{0\leq s\leq
t}\mathbb{E}^{\tilde{\mathbb{Q}}}\Big\{|M(s)|+\int_0^s|X(r)-Y(r)|^2\d
r \Big\} 
\end{equation}
By \eqref{eq12} and \eqref{eq34} we have
\begin{equation*}
\mathbb{E}^{\tilde{\mathbb{Q}}}|X(t)-Y(t)|^2\leq-\beta_1\mathbb{E}^{\tilde{\mathbb{Q}}}\int_0^t|X(s)-Y(s)|^2\d
s+\beta_2\int_0^t\mathbb{E}^{\tilde{\mathbb{Q}}}|h(s)|^2\d s,\ \ \
t\in[0,N],
\end{equation*}
where $\beta_1:=\tilde{\lambda}/(1-\sqrt{\kappa})^2>0$ and
$\beta_2:=4(1+\kappa)\tilde{\sigma}^2/(\epsilon(1-\sqrt{\kappa})^2)$.
This, along with the Gronwall inequality, yields that
\begin{equation}\label{eq35}
\mathbb{E}^{\tilde{\mathbb{Q}}}|X(t)-Y(t)|^2\leq\beta_2\int_0^t
e^{-\beta_1(t-s)}\mathbb{E}^{\tilde{\mathbb{Q}}}|h(s)|^2\d s, \ \ \
t\in[0,N].
\end{equation}
Also by the It\^o formula and the Young inequality we can deduce
from \eqref{eq7} and  \eqref{eq8} that for any  $w_2\in\mathcal
{W}([-\tau,0];\mathbb{R}_+),$  $\epsilon>0$ and
$t\in[k,k+1),k=0,1,\cdots, N-1$,
\begin{equation*}
\begin{split}
|M(t)|^2 &\leq |M(k)|^2-\lambda_1\int_k^t|X(s)-Y(s)|^2\d
s+\lambda_2\int_{-\tau}^0w_2(\theta)\int_k^t|X(s+\theta)-Y(s+\theta)|^2\d
s\d
\theta\\
&\quad+2\tilde{\sigma}\int_k^t|M(s)||h(s)|\d s+2\int_k^t\langle
M(s),\sigma(X_s)-\sigma(Y_s)\d \tilde{W}(s)\rangle\\
&\leq
|M(k)|^2+(\lambda_2+2\kappa\epsilon)\int_{k-\tau}^k|X(s)-Y(s)|^2\d
s+C_1\int_k^t|X(s)-Y(s)|^2\d
s\\
&\quad+\frac{\tilde{\sigma}^2}{\epsilon}\int_k^t|h(s)|^2\d
s+2\int_k^t\langle M(s),\sigma(X_s)-\sigma(Y_s)\d
\tilde{W}(s)\rangle,
\end{split}
\end{equation*}
where $C_1:=-\lambda_1+\lambda_2+2(1+\kappa)\epsilon$. Choosing
$\epsilon>0$ such that $C_1>0$ we have
\begin{equation}\label{eq36}
\begin{split}
\mathbb{E}^{\tilde{\mathbb{Q}}}\Big(\sup_{k\leq t\leq
k+1}|M(t)|^2\Big)&\leq
\mathbb{E}^{\tilde{\mathbb{Q}}}|M(k)|^2+(\lambda_2+2\kappa\epsilon)\mathbb{E}^{\tilde{\mathbb{Q}}}\int_{k-\tau}^k|X(s)-Y(s)|^2\d
s\\
&\quad+C_1\mathbb{E}^{\tilde{\mathbb{Q}}}\int_k^{k+1}|X(s)-Y(s)|^2\d
s\\
&\quad+\frac{\tilde{\sigma}^2}{\epsilon}\mathbb{E}^{\tilde{\mathbb{Q}}}\int_k^{k+1}|h(s)|^2\d
s+J(t),
\end{split}
\end{equation}
where
\begin{equation*}
 J(t):=2\mathbb{E}^{\tilde{\mathbb{Q}}}\Big(\sup_{k\leq t\leq
k+1}\Big|\int_k^t\langle M(s),\sigma(X_s)-\sigma(Y_s)\d
\tilde{W}(s)\rangle\Big|\Big).
\end{equation*}
By the Burkhold-Davis-Gundy inequality and the Young inequality,
together with \eqref{eq38}, it then follows that
\begin{equation}\label{eq42}
\begin{split}
J(t)&\leq6\mathbb{E}^{\tilde{\mathbb{Q}}}\Big(\int_k^{k+1}|
M(s)|^2\|\sigma(X_s)-\sigma(Y_s)\|^2_{HS}\d s\Big)^{1/2}\\
&\leq\frac{1}{2}\mathbb{E}^{\tilde{\mathbb{Q}}}\Big(\sup_{k\leq
t\leq
k+1}|M(t)|^2\Big)+9\lambda_3\mathbb{E}^{\tilde{\mathbb{Q}}}\int_k^{k+1}\int_{-\tau}^0|X(s+\theta)-Y(s+\theta)|^2\d
\theta\d
s\\
&\leq\frac{1}{2}\mathbb{E}^{\tilde{\mathbb{Q}}}\Big(\sup_{k\leq
t\leq
k+1}|M(t)|^2\Big)+9\lambda_3\tau\mathbb{E}^{\tilde{\mathbb{Q}}}\int_{k-\tau}^{k+1}|X(s)-Y(s)|^2\d
s.
\end{split}
\end{equation}
Thus from \eqref{eq36} one has
\begin{equation}\label{eq37}
\begin{split}
\mathbb{E}^{\tilde{\mathbb{Q}}}\Big(\sup_{k\leq t\leq
k+1}|M(t)|^2\Big)&\leq
2\mathbb{E}^{\tilde{\mathbb{Q}}}|M(k)|^2+2(\lambda_2+2\kappa\epsilon)\mathbb{E}^{\tilde{\mathbb{Q}}}\int_{k-\tau}^k|X(s)-Y(s)|^2\d
s\\
&\quad+2C_1\mathbb{E}^{\tilde{\mathbb{Q}}}\int_k^{k+1}|X(s)-Y(s)|^2\d
s+\frac{\tilde{2\sigma}^2}{\epsilon}\mathbb{E}^{\tilde{\mathbb{Q}}}\int_k^{k+1}|h(s)|^2\d s\\
&\quad+18\lambda_3\tau\mathbb{E}^{\tilde{\mathbb{Q}}}\int_{k-\tau}^{k+1}|X(s)-Y(s)|^2\d
s.
\end{split}
\end{equation}
Moreover note from \eqref{eq6} that
\begin{equation*}
\begin{split}
\mathbb{E}^{\tilde{\mathbb{Q}}}\Big(\sup_{k\leq t\leq
k+1}|X(t)-Y(t)|^2\Big)
&\leq\frac{1}{(1-\sqrt{\kappa})^2}\mathbb{E}^{\tilde{\mathbb{Q}}}\Big(\sup_{k\leq
t\leq
k+1}|M(t)|^2\Big)\\
&\quad+\frac{\sqrt{\kappa}}{1-\sqrt{\kappa}}\mathbb{E}^{\tilde{\mathbb{Q}}}\Big(\sup_{k-\tau\leq
t\leq k}|X(t)-Y(t)|^2\Big).
\end{split}
\end{equation*}
 On the other hand, by \eqref{eq7}
\begin{equation*}
\begin{split}
\mathbb{E}^{\tilde{\mathbb{Q}}}
|M(k)|^2\leq2\mathbb{E}^{\tilde{\mathbb{Q}}}|X(k)-Y(k)|^2+2\int_{-\tau}^0w_1(\theta)\mathbb{E}^{\tilde{\mathbb{Q}}}|X(k+\theta)-Y(k+\theta)|^2\d\theta
\end{split}
\end{equation*}
As a result, due to  \eqref{eq37}  we obtain that
\begin{equation*}
\begin{split}
\mathbb{E}^{\tilde{\mathbb{Q}}}\Big(\sup_{k\leq t\leq
k+1}|X(t)-Y(t)|^2\Big)&\leq
C\Big\{\mathbb{E}^{\tilde{\mathbb{Q}}}\Big(\sup_{k-\tau\leq t\leq
k}|X(t)-Y(t)|^2\Big)+\mathbb{E}^{\tilde{\mathbb{Q}}}\int_k^{k+1}|h(s)|^2\d
s\\
&\quad+\mathbb{E}^{\tilde{\mathbb{Q}}}|X(k)-Y(k)|^2+\mathbb{E}^{\tilde{\mathbb{Q}}}\int_{k-\tau}^k|X(s)-Y(s)|^2\d
s\\
&\quad+\int_{-\tau}^0w_1(\theta)\mathbb{E}^{\tilde{\mathbb{Q}}}|X(k+\theta)-Y(k+\theta)|^2\d\theta\Big\}\\
&\quad+C\mathbb{E}^{\tilde{\mathbb{Q}}}\int_k^{k+1}|X(s)-Y(s)|^2\d
s.
\end{split}
\end{equation*}
An application of the Gronwall inequality further implies that
\begin{equation*}
\begin{split}
\mathbb{E}^{\tilde{\mathbb{Q}}}\Big(\sup_{k\leq t\leq
k+1}|X(t)-Y(t)|^2\Big)&\leq
C\Big\{\mathbb{E}^{\tilde{\mathbb{Q}}}\Big(\sup_{k-\tau\leq t\leq
k}|X(t)-Y(t)|^2\Big)+\mathbb{E}^{\tilde{\mathbb{Q}}}\int_k^{k+1}|h(s)|^2\d
s\\
&\quad+\mathbb{E}^{\tilde{\mathbb{Q}}}
|X(k)-Y(k)|^2+\mathbb{E}^{\tilde{\mathbb{Q}}}\int_{k-\tau}^k|X(s)-Y(s)|^2\d
s\\
&\quad+\int_{-\tau}^0w_1(\theta)\mathbb{E}^{\tilde{\mathbb{Q}}}|X(k+\theta)-Y(k+\theta)|^2\d
\theta\Big\}.
\end{split}
\end{equation*}
Thus for any $1\leq M\leq N$ one has
\begin{equation*}
\begin{split}
\sum_{k=0}^{M-1}\mathbb{E}^{\tilde{\mathbb{Q}}}\Big(\sup_{k\leq
t\leq k+1}|X(t)-Y(t)|^2\Big)&\leq
C\Big\{\sum_{k=0}^{M-1}\mathbb{E}^{\tilde{\mathbb{Q}}}\Big(\sup_{k-\tau\leq
t\leq
k}|X(t)-Y(t)|^2\Big)+\mathbb{E}^{\tilde{\mathbb{Q}}}\int_0^M|h(s)|^2\d
s\\
&\quad+\sum_{k=0}^{M-1}\mathbb{E}^{\tilde{\mathbb{Q}}}
|X(k)-Y(k)|^2+\sum_{k=0}^{M-1}\mathbb{E}^{\tilde{\mathbb{Q}}}\int_{k-\tau}^k|X(s)-Y(s)|^2\d
s\\
&\quad+\int_{-\tau}^0w_1(\theta)\sum_{k=0}^{M-1}\mathbb{E}^{\tilde{\mathbb{Q}}}|X(k+\theta)-Y(k+\theta)|^2\d
\theta\Big\}.
\end{split}
\end{equation*}
Note from \eqref{eq35} that
\begin{equation*}
\begin{split}
\sum_{k=0}^{M-1}\mathbb{E}^{\tilde{\mathbb{Q}}}\int_{k-\tau}^k|X(s)-Y(s)|^2\d
s &\leq
\beta_2\sum_{k=0}^{M-1}\int_{k-\tau}^k\int_0^se^{-\beta_1(s-r)}\mathbb{E}^{\tilde{\mathbb{Q}}}|h(r)|^2\d
r\d s\\
 &=
\beta_2\sum_{k=0}^{M-1}\int_{k-\tau}^k\int_0^se^{-\beta_1(k-r)}e^{-\beta_1(s-k)}\mathbb{E}^{\tilde{\mathbb{Q}}}|h(r)|^2\d
r\d s\\
&\leq
\beta_2e^{\beta_1\tau}\tau\sum_{k=0}^{M-1}\int_0^ke^{-\beta_1(k-s)}\mathbb{E}^{\tilde{\mathbb{Q}}}|h(s)|^2\d
s,
\end{split}
\end{equation*}
and due to $w_1\in\mathcal {W}([-\tau,0];\mathbb{R}_+)$
\begin{equation*}
\begin{split}
&\int_{-\tau}^0w_1(\theta)\sum_{k=0}^{M-1}\mathbb{E}^{\tilde{\mathbb{Q}}}|X(k+\theta)-Y(k+\theta)|^2\d
\theta\\
&\leq
\beta_2\int_{-\tau}^0w_1(\theta)\sum_{k=0}^{M-1}\int_0^{k+\theta}
e^{-\beta_1(k+\theta-s)}\mathbb{E}^{\tilde{\mathbb{Q}}}|h(s)|^2\d
s\d \theta\\
&\leq \beta_2e^{\beta_1\tau}\sum_{k=0}^{M-1}\int_0^{k}
e^{-\beta_1(k-s)}\mathbb{E}^{\tilde{\mathbb{Q}}}|h(s)|^2\d s.
\end{split}
\end{equation*}
Then, carrying out a similar procedure to that of
\cite[\mbox{Inquality} (3.6), p363]{WZ04}, we can derive that for
any $1\leq M\leq N$
\begin{equation*}
\sum_{k=0}^{M-1}\mathbb{E}^{\tilde{\mathbb{Q}}}\Big(\sup_{k\leq
t\leq k+1}|X(t)-Y(t)|^2\Big)\leq
C\Big\{\sum_{k=0}^{M-2}\mathbb{E}^{\tilde{\mathbb{Q}}}\Big(\sup_{k\leq
t\leq
k+1}|X(t)-Y(t)|^2\Big)+\mathbb{E}^{\tilde{\mathbb{Q}}}\int_0^N|h(s)|^2\d
s\Big\},
\end{equation*}
and  the claim \eqref{eq13} follows by a deduction argument.

({\bf2}) Observing the proceeding proof we can also deduce that
\begin{equation*}
\mathbb{E}^{\tilde{\mathbb{Q}}}|X(t)-Y(t)|^2 \leq
C\int_0^te^{C(t-s)}\mathbb{E}^{\tilde{\mathbb{Q}}}|h(s)|^2\d s, \ \
\ t\in[0,T].
\end{equation*}
By changing the  integral order it follows that
\begin{equation*}
\begin{split}
\int_0^T\mathbb{E}^{\tilde{\mathbb{Q}}}|X(t)-Y(t)|^2 \d t&\leq
C\int_0^T\int_0^te^{C(t-s)}\mathbb{E}^{\tilde{\mathbb{Q}}}|h(s)|^2\d
s\d t\\
&=C\int_0^T\mathbb{E}^{\tilde{\mathbb{Q}}}|h(s)|^2\int_s^Te^{c_3(t-s)}\d
t\d s\\
&\leq C\int_0^T\mathbb{E}^{\tilde{\mathbb{Q}}}|h(s)|^2\d s,
\end{split}
\end{equation*}
and then $\mathbb{P}_\xi\in T_2(C)$ for some $C>0$ under the metric
$\d_{L^2}$.  The proof is therefore complete.
\end{proof}

\begin{rem}
{\rm The processes $X(t)$ we consider throughout the paper is not a
Markov process since $G,b$ and $\sigma$ are dependent on the past
history of the solution process.}
\end{rem}

Observing the argument of Theorem \ref{Theorem 1} we can also obtain
the following result.
\begin{thm}\label{Theorem 4}
{\rm Let the conditions in Theorem \ref{Theorem 1} hold. Then
$\mathbb{P}_\xi\in T_2(C)$ for some constant $C>0$, which is
dependent only on $\kappa,\lambda_i,i=1,2,3,\tilde{\sigma}$ and the
universal constant from the Burkhold-Davis-Gundy inequality, under
uniform metric
\begin{equation*}
\d_{\infty,2}(\gamma_1,\gamma_2):=\sup_{0\leq t\leq
T}|\gamma_1(t)-\gamma_2(t)|,\ \ \ \gamma_1,\gamma_2\in\mathcal {X}.
\end{equation*}

}
\end{thm}

 In what follows we shall show $\mathbb{P}_\xi\in T_2(C)$ for some $C>0$ with respect
to  the uniform metric $\d_{\infty,2}$  under weaker conditions than
\eqref{eq7}-\eqref{eq8}.  For $ \xi\in\mathscr{C}$, denote  $\|\xi\|_\8=  \sup_{-\tau\leq\theta\leq0}|\xi(\theta)|.$ Assume that $G(0)=0$ and there exists
$\kappa\in(0,1)$ such that
\begin{equation}\label{eq14}
|G(\xi)-G(\eta)|\leq\kappa\|\xi-\eta\|_\infty, \ \ \ \xi,\eta\in\mathscr{C}.
\end{equation}
Assume also that $b$ and $\sigma$ are locally Lipschizian and there
exists $\tilde \lambda_1, \tilde\lambda_2\geq0$ such that for arbitrary
$\xi,\eta\in\mathscr{C}$
\begin{equation}\label{eq15}
2\langle\xi(0)-G(\xi),b(\xi)\rangle+\|\sigma(\xi)\|_{HS}^2
\leq\tilde\lambda_1(1+\|\xi\|^2_\infty)
\end{equation}
and
\begin{equation}\label{eq16}
2\langle\xi(0)-G(\xi)-(\eta(0)-G(\eta)),b(\xi)-b(\eta)\rangle+\|\sigma(\xi)-\sigma(\eta)\|_{HS}^2
\leq\tilde\lambda_2\|\xi-\eta\|^2_\infty.
\end{equation}

\begin{thm}\label{Theorem 2}
{\rm Let \eqref{eq14}-\eqref{eq16} hold. Assume further that there
exists $\tilde\lambda_3\geq0$ such that
\begin{equation}\label{eq39}
\|\sigma(\xi)-\sigma(\eta)\|_{HS}^2\leq\tilde\lambda_3\|\xi(\theta)-\eta(\theta)\|^2_\infty,
\ \ \ \xi,\eta\in\mathscr{C},
\end{equation}
and $\sigma$ is bounded by
$\tilde{\sigma}:=\sup_{\xi\in\mathscr{C}}\|\sigma(\xi)\|_{HS}<\infty$.
Then  $\mathbb{P}_\xi\in T_2(C)$ for some $C>0$ on the metric space
$\mathcal {X}$ with respect to the metric $\d_{\infty,2}$. }
\end{thm}

\begin{proof}
Note that, under \eqref{eq14}-\eqref{eq16}, Eq. \eqref{eq1} has a
unique solution $\{X(t,\xi)\}_{t\in[-\tau,T]}$. We also point out
that the argument is similar to that of Theorem \ref{Theorem 1},
while  a sketch of the proof is provided for the 
completeness and highlight some  differences. Let $\tilde{\mathbb{Q}}$
be defined by \eqref{eq17}. Adopting a similar procedure to that of
\eqref{eq18}, we have
\begin{equation}\label{eq5}
(W_{2,\d}(\mathbb{Q},\mathbb{P}_\xi))^2\leq\mathbb{E}^{\tilde{\mathbb{Q}}}(\d_{\infty,2}^2(X,Y)).
\end{equation}
Thus by \eqref{eq4} and \eqref{eq5} to verify that
$\mathbb{P}_\xi\in T_2(C)$ for some $C>0$ under the metric
$\d_{\infty,2}$ it is sufficient to verify that
\begin{equation}\label{eq11}
\mathbb{E}^{\tilde{\mathbb{Q}}}\Big(\sup_{0\leq t\leq
T}|X(t)-Y(t)|^2\Big)\leq
C\mathbb{E}^{\tilde{\mathbb{Q}}}\int_0^T|h(t)|^2\d t.
\end{equation}
Let $M(t):=X(t)-G(X_t)-(Y(t)-G(Y_t)),t\in[0,T]$. Recalling the
inequality \eqref{eq6} and noting that $X_0=Y_0=\xi$, we then obtain
from \eqref{eq14} that
\begin{equation*}
\begin{split}
\mathbb{E}^{\tilde{\mathbb{Q}}}\Big(\sup_{0\leq t\leq T}
|X(t)-Y(t)|^2\Big)&\leq\frac{1}{1-\kappa}\mathbb{E}^{\tilde{\mathbb{Q}}}\Big(\sup_{0\leq t\leq T}|M(t)|^2\Big)+\frac{1}{\kappa}\mathbb{E}^{\tilde{\mathbb{Q}}}\Big(\sup_{0\leq t\leq T}|G(X_t)-G(Y_t)|^2\Big)\\
&\leq\frac{1}{1-\kappa}\mathbb{E}^{\tilde{\mathbb{Q}}}\Big(\sup_{0\leq
t\leq
T}|M(t)|^2\Big)+\kappa\mathbb{E}^{\tilde{\mathbb{Q}}}\Big(\sup_{0\leq
t\leq T}|X(t)-Y(t)|^2\Big).
\end{split}
\end{equation*}
This implies that
\begin{equation}\label{eq10}
\mathbb{E}^{\tilde{\mathbb{Q}}}\Big(\sup_{0\leq t\leq
T}|X(t)-Y(t)|^2\Big)\leq\frac{1}{(1-\kappa)^2}\mathbb{E}^{\tilde{\mathbb{Q}}}\Big(\sup_{0\leq
t\leq T}|M(t)|^2\Big).
\end{equation}
 Then, applying the It\^o formula and the Burkhold-Davis-Gundy inequality, we can derive from
\eqref{eq14}, \eqref{eq16}and \eqref{eq39} that
\begin{equation*}
\begin{split}
\mathbb{E}^{\tilde{\mathbb{Q}}}\Big(\sup_{0\leq t\leq
T}|M(t)|^2\Big)
&\leq(\tilde \lambda_2+2(1+\kappa^2))\int_0^T\mathbb{E}^{\tilde{\mathbb{Q}}}\Big(\sup\limits_{0\leq
s\leq t}|X(s)-Y(s)|^2\Big)\d s+\tilde{\sigma}^2\int_0^T
\mathbb{E}^{\tilde{\mathbb{Q}}}|h(t)|^2\d t\\
&\quad+6\mathbb{E}^{\tilde{\mathbb{Q}}}\Big(\sup_{0\leq t\leq
T}|M(t)|^2\int_0^T\|\sigma(X_s)-\sigma(Y_s)\|_{HS}^2
\d s\Big)^{\frac{1}{2}}\\
&\leq\Big(\tilde\lambda_2+2(1+\kappa^2)+18\tilde\lambda_3^2\Big)\int_0^T\mathbb{E}^{\tilde{\mathbb{Q}}}\Big(\sup\limits_{0\leq
s\leq t}|X(s)-Y(s)|^2\Big)\d s\\
&\quad+\tilde{\sigma}^2\int_0^T
\mathbb{E}^{\tilde{\mathbb{Q}}}|h(t)|^2\d t
+\frac{1}{2}\mathbb{E}^{\tilde{\mathbb{Q}}}\Big(\sup_{0\leq t\leq
T}|M(t)|^2\Big),
\end{split}
\end{equation*}
where we have also used the boundedness of $\sigma$ and  the Young
inequality in the last step. It then follows that
\begin{equation*}
\mathbb{E}^{\tilde{\mathbb{Q}}}\Big(\sup_{0\leq t\leq
T}|M(t)|^2\Big)\leq
C\int_0^T\mathbb{E}^{\tilde{\mathbb{Q}}}\Big(\sup\limits_{0\leq
s\leq t}|X(s)-Y(s)|^2\Big)\d s+2\tilde{\sigma}^2\int_0^T
\mathbb{E}^{\tilde{\mathbb{Q}}}|h(t)|^2\d t.
\end{equation*}
Thus from \eqref{eq10} we get
\begin{equation*}
\begin{split}
\mathbb{E}^{\tilde{\mathbb{Q}}}\Big(\sup_{0\leq t\leq
T}|X(t)-Y(t)|^2\Big)&\leq\frac{C}{(1-\kappa)^2}\int_0^T\mathbb{E}^{\tilde{\mathbb{Q}}}\Big(\sup\limits_{0\leq
s\leq t}|X(s)-Y(s)|^2\Big)\d s\\
&\quad+\frac{2\tilde{\sigma}^2}{(1-\kappa)^2}\int_0^T
\mathbb{E}^{\tilde{\mathbb{Q}}}|h(t)|^2\d t,
\end{split}
\end{equation*}
and  due to the Gronwall inequality
\begin{equation*}
\mathbb{E}^{\tilde{\mathbb{Q}}}\Big(\sup_{0\leq t\leq
T}|X(t)-Y(t)|^2\Big)\leq C\int_0^T
\mathbb{E}^{\tilde{\mathbb{Q}}}|h(t)|^2\d t.
\end{equation*}
Consequently the statement \eqref{eq11} follows and thus
$\mathbb{P}_\xi\in T_2(C)$ holds for some $C>0$, as required.
\end{proof}

\begin{rem}
{\rm If, in \eqref{eq8}, $G\equiv0,\tau=0$, $\lambda_1>0$ and
$\lambda_2=0$, Theorem \ref{Theorem 1} ({\it under the metric
$\d_{\infty,1}$}) becomes \cite[Theorem 2.1]{WZ04}, where the
constant $C$ in  \cite[Theorem 2.1]{WZ04} is independent of time
$T=N$. While for the functional case,  the constant $C>0$ in Theorem
\ref{Theorem 1} is dependent on time $T=N$ even for $\lambda_1>0$
(in \eqref{eq8}), which demonstrate the differences between SDEs
without memory and functional SDEs. Nevertheless  {\it under the
metric} $\d_{\infty,2}$ Theorem \ref{Theorem 4} shows that the
constant $C>0$ such that TCI holds can be chosen independent of $T$.
 Moreover, following the argument of Theorem \ref{Theorem 1}, our
main result can also be generalized to the case of neutral
functional with infinity delay, which will be reported in the future
paper. }
\end{rem}

\begin{rem}
{\rm For time-inhomogenous diffusions, \"Ust\"unel \cite[Proposition
1]{U10} and Pal \cite[Theorem 5]{P11} verify that $\mathbb{P}_\xi\in
T_2(C)$ with respect to $\d_{\infty,2}$ for the dissipative case and
infinite time horizon case respectively. While in this section we
discuss the Talagrand's $T_2$-transportation inequality with respect
to two uniform metrics $\d_{\infty,1}$ and $\d_{\infty,2}$ for a
class of {\it neutral functional } SDEs, and in particular some
techniques have been developed to deal with  the difficulties caused
by the neutral term and the time-lag. }
\end{rem}

\section{TCI for Neutral Functional SPDEs}
 In this section we proceed to discuss the
TCIs for the laws of a class of neutral functional SPDEs in
infinite-dimensional setting.
 Let $(H,\langle\cdot,\cdot\rangle_H,\|\cdot\|_H)$ be a real
separable Hilbert space,  and  $(W(t))_{t\geq0}$   a cylindrical
Wiener process on $H$ with respect to  a filtered complete
probability space $(\Omega, \scr {F}, \{\scr
{F}_t\}_{t\geq0},\mathbb{P})$. Let $\scr {L}(H)$ and $\scr
{L}_{HS}(H)$ be the spaces of all linear bounded operators and
Hilbert-Schmidt operators on $H$ respectively. Denote by $\|\cdot\|$
and $\|\cdot\|_{HS}$  the operator norm and the Hilbert-Schmidt norm
respectively. Fix $\tau>0$ and let $\mathscr{C}:=\mathcal
{C}([-\tau,0]; H)$, the space of  continuous functions
$f:[-\tau,0]\mapsto H$, equipped with a uniform norm
$\|f\|_\infty:=\sup_{-\tau\leq\theta\leq0}\|f(\theta)\|_H$.

For $T>0$ consider semi-linear neutral functional SPDE on $H$
\begin{equation}\label{eq19}
\begin{cases} \d[X(t)+G(X_t)]=[AX(t)+b(X_t)]\d t+\sigma(X_t)\d
W(t),\ \ \
t\in[0,T],\\
X_0=\xi\in\mathscr{C}.
\end{cases}
\end{equation}

We assume that
\begin{enumerate}
\item[\textmd{(A1)}] $(A,\scr {D}(A))$ is a linear
operator on $H$ generating an analytic $C_0$-semigroup
$(\e^{tA})_{t\geq0}$ such that $\|e^{tA}\|\leq Me^{\nu t}$ for some
$M,\nu>0$.
\item[\textmd{(A2)}]
$b:\mathscr{C}\rightarrow H$ and there is $\rho_1>0$ such that
$\|b(\xi)-b(\eta)\|_H\leq\rho_1\|\xi(\theta)-\eta(\theta)\|_\infty,\xi,\eta\in\mathscr{C}$.
\item[\textmd{(A3)}] $ \sigma: H\rightarrow \scr {L}(H)$ such that
$e^{tA}\sigma(0)\in\scr {L}_{HS}(H)$ for any $t\in[0,T$] and
$\int_0^T s^{-2\beta}\|e^{s A}\si(0)\|_{HS}^2\d s<\infty$ for some
$\beta\in (0,\frac{1}{2})$. Moreover assume that there exists
$\rho_2>0$ such that $\|\sigma(\xi)-\sigma(\eta)\|_{HS}\leq
\rho_2\|\xi(\theta)-\eta(\theta)\|_\infty,\xi,\eta\in\mathscr{C}$.
\item[\textmd{(A4)}]  $G:\mathscr{C}\rightarrow H$ such that
$G(0)=0$ and there exist $\alpha\in[0,1]$ and $\rho_3>0$ such
that $\|(-A)^\alpha G(\xi)-(-A)^\alpha
G(\eta)\|_H\leq\rho_3\|\xi(\theta)-\eta(\theta)\|_\infty$, where
$(-A)^\alpha$ is the fractional power of $-A$.
\item[\textmd{(A5)}] $\rho_3\Big(\|(-A)^{-\alpha}\|+M_{1-\alpha}\int_0^T\frac{e^{\nu
t}}{t^{1-\alpha}}\d t\Big)<1$, where $(-A)^{-\alpha}$ is the inverse
of $-A$ and $M_{1-\alpha}$ is the positive constant in Lemma
\ref{Lemma 3} below.
\end{enumerate}

\begin{rem}
{\rm We remark from $(A3)$ that $\sigma(\xi)-\sigma(\eta)\in\scr
{L}_{HS}(H)$ while $\sigma$ need not be Hilbert-Schmidt. }
\end{rem}

 Note that $\int_0^T s^{-2\beta}\|e^{s A}\si(0)\|_{HS}^2\d
s<\infty$ remains true by replacing $\beta$ with a smaller positive
number. So, we may take in  (A3)   $\beta\in (0, \ff 1 p)$ for
$p>2$. Denote by $\mathscr{H}_p$ the Banach space of all $H$-valued
continuous adapted processes $Y$ defined on the time interval
$[-\tau,T]$ such that $X(t)=\xi(t),t\in[-\tau,0]$, and
\begin{equation*}
\|X\|_p:=\Big(\mathbb{E}\sup_{t\in[-\tau,T]}\|X(t)\|^p_H\Big)^{\frac{1}{p}}<\infty.
\end{equation*}
By the classical Banach-fixed-point-theorem approach,  Eq.
\eqref{eq19} admits a unique mild solution
$\{X(t,\xi)\}_{t\in[0,T]}$.  That is, for any $\xi\in\mathscr{C}$
there exists a unique $H$-valued adapted process
$\{X(t,\xi)\}_{t\in[0,T]}$, which is continuous in
$L^2(\Omega,\mathbb{P})$, such that
\begin{equation*}
\begin{split}
X(t)&=e^{tA}[\xi(0)+G(\xi)]-G(X_t)-\int_0^tAe^{(t-s)A}G(X_s)\d s\\
&+\int_0^te^{(t-s)A}b(X_s)\d s+\int_0^te^{(t-s)A}\sigma(X_s)\d W(s).
\end{split}
\end{equation*}

\begin{rem}
{\rm Since $\sigma$ need not be Hilbert-Schmidt, the
Burkhold-Davis-Gundy inequality \cite[Proposition, p196]{DZ} may not
hold for $p=2$. Therefore the mild solution is shown by the fixed
point theorem on Banach space $\mathscr{H}_p,p>2$. }
\end{rem}

The following  lemma \cite[Theorem 6.13]{P92} is vital to
deal with the neutral term $G$.
\begin{lem}\label{Lemma 3}
{\rm Under $(A1)$, for any $\beta\in(0,1]$ and  $x\in \mathcal
{D}((-A)^{\beta})$
\begin{equation*}
e^{tA}(-A)^{\beta}x=(-A)^{\beta}e^{tA}x
\end{equation*}
and there exists  $M_{\beta}>0$ such that for any $t>0$
\begin{equation*}
\parallel(-A)^{\beta}e^{tA}\parallel\leq M_{\beta }t^{-\beta}e^{\nu
t}.
\end{equation*}
}
\end{lem}

\begin{thm}\label{Theorem 3}
{\rm Let $(A1)-(A5)$ hold and  $\mathbb{P}_\xi$ be the law of
$X(\cdot,\xi)$, solution process of Eq. \eqref{eq19}. Assume further
that $\sigma$ is bounded by
$\tilde{\sigma}:=\sup_{\xi\in\mathscr{C}}\|\sigma(\xi)\|$. Then
$\mathbb{P}_\xi\in T_2(C)$ for some $C>0$ on the metric space
$\bar{\mathcal {X}}:=\mathcal {C}([0,T];H)$ with respect to the
metric
\begin{equation*}
\d_\infty(\gamma_1,\gamma_2):=\sup_{0\leq t\leq
T}\|\gamma_1(t)-\gamma_2(t)\|_H,\ \ \ \gamma_1,\gamma_2\in\bar{\mathcal{X}}.
\end{equation*}}
\end{thm}

\begin{proof}
We should point out that  the former parts of the argument is
similar to that of Theorem \ref{Theorem 1}, while
due to the unboundedness of infinitesimal generator $A$ and the
appearance of neutral term $G$, the It\^o formula  is not
unavailable although $\sigma(\xi)-\sigma(\eta)\in\scr
{L}_{HS}(H),\xi,\eta\in\mathscr{C}$. In what follows we will use the theory of the semigroup.
Let $\mathbb{P}_\xi$ be the law of
$X(\cdot,\xi)$ on $\mathcal {X}$ and $\tilde{\mathbb{Q}}$  defined
by \eqref{eq17}. By the martingale representation theorem
\cite[Theorem 8.2, p220]{DZ} we can also deduce that there exists a
predictable process $h\in H$ with $\int_0^T\|h(s)\|^2_H\d s<\infty$,
$\mathbb{P}$-\mbox{a.s.}, such that
\begin{equation*}
\mbox{\bf H}(\tilde{\mathbb{Q}}|\mathbb{P})=\mbox{\bf
H}(\mathbb{Q}|\mathbb{P}_\xi)=\frac{1}{2}\mathbb{E}^{\tilde{\mathbb{Q}}}\int_0^T\|h(t)\|^2_H\d
t.
\end{equation*}
Due to the Girsanov theorem
\begin{equation*}
\tilde{W}(t):=W(t)-\int_0^th(s)\d s,\ \ \ t\in[0,T],
\end{equation*}
is a Brownian motion with respect to $\{\mathcal {F}_t\}_{t\geq0}$
on the probability space $(\Omega,\mathcal {F},\tilde{\mathbb{Q}})$.
Then,  under the measure $\tilde{\mathbb{Q}}$, the process
$\{X(t,\xi)\}_{t\in[0,T]}$ satisfies
\begin{equation}\label{eq20}
\begin{cases}
\d[X(t)+G(X_t)]=[AX(t)+b(X_t)+\sigma(X_t)h(t)]\d t+\sigma(X_t)\d \tilde{W}(t),\\
X_0=\xi.
\end{cases}
\end{equation}
Let $\{Y(t,\xi)\}_{t\in[0,T]}$ be the solution of the following
equation
\begin{equation}\label{eq21}
\begin{cases}
\d[Y(t)+G(Y_t)]=[AY(t)+b(Y_t)]\d t+\sigma(Y_t)\d \tilde{W}(t),\\
Y_0=\xi.
\end{cases}
\end{equation}
By the uniqueness, under $\tilde{\mathbb{Q}}$ the law of $Y(\cdot)$
is $\mathbb{P}_\xi$. Thus $(X,Y)$ under $\tilde{\mathbb{Q}}$ is a
coupling of $(\mathbb{Q},\mathbb{P}_\xi)$, and
\begin{equation*}
(W_{2,\d}(\mathbb{Q},\mathbb{P}_\xi))^2\leq\mathbb{E}^{\tilde{\mathbb{Q}}}(\d_\infty(X,Y)).
\end{equation*}
Thus we only need to show
\begin{equation}\label{eq25}
\mathbb{E}^{\tilde{\mathbb{Q}}}\Big(\sup_{0\leq t\leq
T}\|X(t)-Y(t)\|^2_H\Big)\leq
C\mathbb{E}^{\tilde{\mathbb{Q}}}\int_0^T\|h(t)\|^2_H\d t.
\end{equation}
Note from \eqref{eq20} and \eqref{eq21}, together with the
inequality \eqref{eq6}, that for any $\epsilon\in(0,1)$
\begin{equation}\label{eq26}
\begin{split}
\|X(t)-Y(t)\|_H^2&\leq \frac{1}{\epsilon}\Big\{\|G(Y_t)-G(X_t)\|_H+\Big\|\int_0^tAe^{(t-s)A}[G(Y_s)-G(X_s)]\d s\Big\|_H\Big\}^2\\
&\quad+\frac{3}{1-\epsilon}\Big\{\Big\|\int_0^te^{(t-s)A}[b(X_s)-b(Y_s)]\d
s\Big\|_H^2\\
&\quad+\Big\|\int_0^te^{(t-s)A}\sigma(X_s)h(s)\d
s\Big\|_H^2\Big\}\\
&\quad+\frac{3}{1-\epsilon}\Big\|\int_0^te^{(t-s)A}[\sigma(X_s)-\sigma(Y_s)]\d
\tilde{W}(s)\Big\|_H^2\\
&:=I_1(t)+I_2(t)+I_3(t), \ \ \ t\in[0,T],
\end{split}
\end{equation}
where we have also used the fundamental inequality
$(a+b+c)^3\leq3(a^2+b^2+c^2)$ for $a,b,c\in\mathbb{R}$. Note
 that $(-A)^{-\alpha}$ is bounded by \cite[Lemma 6.3, p71]{P92},
 and,
 in the light of \cite[Theorem 6.8, p72]{P92},
\begin{equation*}
(-A)^{\alpha+\beta}x=(-A)^\alpha\cdot(-A)^\beta x
\end{equation*}
for  $x\in\mathcal {D}((-A)^\gamma)$, the domain of $(-A)^\gamma$,
with $\gamma:=\max\{\alpha,\beta,\alpha+\beta\}$,
$\alpha,\beta\in\mathbb{R}$. By (A4), it follows from Lemma
\ref{Lemma 3} that
\begin{equation*}
\begin{split}
\sup_{0\leq t\leq T}I_1(t)&=\frac{1}{\epsilon}\sup_{0\leq t\leq T}\Big\{\|(-A)^{-\alpha}((-A)^{\alpha}G(Y_t)-(-A)^{\alpha}G(X_t))\|_H\\
&\quad+\Big\|\int_0^t(-A)e^{(t-s)A}(-A)^{-\alpha}[(-A)^{\alpha}G(Y_s)-(-A)^{\alpha}G(X_s)]\d
s\Big\|_H\Big\}^2\\
&=\frac{1}{\epsilon}\sup_{0\leq t\leq T}\Big\{\|(-A)^{-\alpha}((-A)^{\alpha}G(Y_t)-(-A)^{\alpha}G(X_t))\|_H\\
&\quad+\Big\|\int_0^tA^{1-\alpha}e^{(t-s)A}[(-A)^{\alpha}G(Y_s)-(-A)^{\alpha}G(X_s)]\d
s\Big\|_H\Big\}^2\\
&\leq\frac{1}{\epsilon}\Big\{\rho_3\Big(\|(-A)^{-\alpha}\|+\int_0^T\|A^{1-\alpha}e^{tA}\|\d
t\Big)\sup_{0\leq t\leq T}\|X(t)-Y(t)\|_H\Big\}^2\\
&\leq\frac{1}{\epsilon}\Big\{\rho_3\Big(\|(-A)^{-\alpha}\|+M_{1-\alpha}\int_0^T\frac{e^{\nu
t}}{t^{1-\alpha}}\d t\Big)\sup_{0\leq t\leq
T}\|X(t)-Y(t)\|_H\Big\}^2,
\end{split}
\end{equation*}
where $X(t)=Y(t),t\in[-\tau,0]$. Taking
$\epsilon=\rho_2\Big(\|(-A)^{-\alpha}\|+M_{1-\alpha}\int_0^T\frac{e^{\nu
t}}{t^{1-\alpha}}\d t\Big)$ we obtain from (A5) that
\begin{equation*}
\sup_{0\leq t\leq T}I_1(t)\leq\epsilon\sup_{0\leq t\leq
T}\|X(t)-Y(t)\|_H^2.
\end{equation*}
Thus due to \eqref{eq26}
\begin{equation}\label{eq22}
\mathbb{E}\Big(\sup_{0\leq t\leq
T}\|X(t)-Y(t)\|_H^2\Big)\leq\frac{1}{1-\epsilon}\mathbb{E}\Big(\sup_{0\leq
t\leq T}I_2(t)\Big)+\frac{1}{1-\epsilon}\mathbb{E}\Big(\sup_{0\leq
t\leq T}I_3(t)\Big).
\end{equation}
Next, by the H\"older inequality, (A2) and the boundedness of
$\sigma$, we have
\begin{equation}\label{eq23}
\begin{split}
\mathbb{E}\Big(\sup_{0\leq t\leq T}I_2(t)\Big)
&\leq\frac{3T}{1-\epsilon}\Big\{\rho_1^2\tilde{M}^2\int_0^T\mathbb{E}\Big(\sup_{0\leq
s\leq t}\|X(s)-Y(s)\|_H^2\Big)\d
t\\
&\quad+\tilde{M}^2\tilde{\sigma}^2\int_0^T\|h(s)\|^2_H\d s\Big\},
\end{split}
\end{equation}
where $\tilde{M}:=M\sup_{t\in[0,T]}\|e^{tA}\|$. Furthermore by the
Burkhold-Davis-Gundy inequality and (A3) there exists $C>0$ such
that
\begin{equation}\label{eq24}
\begin{split}
\mathbb{E}\Big(\sup_{0\leq t\leq T}I_3(t)\Big)
&\leq\frac{3C\rho_2^2}{1-\epsilon}\int_0^T\mathbb{E}\Big(\sup_{0\leq
s\leq t}\|X(s)-Y(s)\|_H^2\Big)\d t.
\end{split}
\end{equation}
Then \eqref{eq25} follows by substituting \eqref{eq23} and
\eqref{eq24} into \eqref{eq22} and applying the Gronwall inequality,
and the proof is therefore complete.
\end{proof}

To demonstrate the applications of Theorem \ref{Theorem 3}, we
 give an illustrative example motivated by \cite[Example
4.1]{HH98}.
\begin{exa}
{\rm Let $H:=L^2([0,\pi])$, $A:=\frac{\partial^2}{\partial x^2}$
with the domain $ \mathcal {D}:=H^2(0,\pi)\cap H^1_0(0,\pi)$, and
$\{W(t,x),t\in[0,T],x\in[0,\pi]\}$ be a Brownian sheet defined on a
completed probability space $(\Omega,\mathcal {F},\mathbb{P})$,
i.e., it is a centered Gaussian random field with the covariance
$\mathbb{E}(W(t,x)W(s,y))=(t\wedge s)(x\wedge y)$ for $s,t\in[0,T]$
and $x,y\in[0,\pi]$. Let $\phi:\mathbb{R}\mapsto\mathbb{R}$ be
Lipschitzian, i.e., there exists $L>0$ such that
$|\phi(x)-\phi(y)|\leq L|x-y|,x,y\in\mathbb{R}$. Assume further that
$\varphi:[-\tau,0]\times[0,\pi]\times[0,\pi]\mapsto\mathbb{R}$ is
measurable such that
$\varphi(\cdot,\cdot,0)=\varphi(\cdot,\cdot,\pi)=0$ and
\begin{equation}\label{eq04}
N:=\int_{-\tau}^0\int_0^\pi\int_0^\pi\Big(\frac{\partial}{\partial
x}\varphi(\theta,\xi,x)\Big)^2\d \xi\d x\d \theta<\infty.
\end{equation}
Consider neutral functional SPDE
\begin{equation}\label{eq01}
\begin{split}
\frac{\partial}{\partial t}\Big[u(t,x)&+\int_0^\pi\int_{-\tau}^0\varphi(\theta,\xi,x)u(t+\theta,\xi)\d\theta\d\xi\Big]\\
&=\Big\{\frac{\partial^2}{\partial
x^2}u(t,x)+\phi\Big(\int_{-\tau}^0u(t+\theta,x)\d\theta\Big)\Big\}+\frac{\partial^2W}{\partial
t\partial x} (t,x)
\end{split}
\end{equation}
with the Dirichlet boundary condition
\begin{equation*}
X(t,0)=X(t,\pi)=0,\ \ \ t\in[0,T],
\end{equation*}
and the initial condition
\begin{equation*}
X(\theta,x)=\psi(\theta,x),\ \ \ \theta\in[-\tau,0],\ \ \
x\in[0,\pi].
\end{equation*}
Recall that $e_n(x):=(2/\pi)^{1/2}\sin nx,
n\in\mathbb{N},x\in[0,\pi]$, is a complete orthonormal system of
$H$, and that the eigenvector of A with eigenvalue $-n^2$, i.e.,
$Ae_n=-n^2e_n$. Then we have
\begin{equation*}
W(t)(x):=W(t,x)=\sum\limits_{n=1}^\infty e_n(x)\int_0^t\int_0^\pi
e_n(y)W(\d s,\d y),
\end{equation*}
which is a cylindrical Wiener process on $H$. For $t\in[0,T]$ and
$x\in[0,\pi]$, let
\begin{equation*}
X(t)(x):=u(t,x),\ \ \
G(X_t)(x):=\int_0^\pi\int_{-\tau}^0\varphi(\theta,\xi,x)X(t+\theta,\xi)\d\theta\d\xi
\end{equation*}
and
\begin{equation*}
b(X_t)(x):=\phi\Big(\int_{-\tau}^0X(t+\theta,x)\d\theta\Big).
\end{equation*}
Then Eq. \eqref{eq01} can be rewritten in the  form
\eqref{eq19}. Observe that $A$ generates a strongly continuous
semigroup $\{e^{tA}\}_{t\in[0,T]}$, which is compact, analytic and
self-adjoint, and
\begin{equation}\label{eq02}
e^{tA}\xi=\sum\limits_{n=1}^\infty e^{-n^2t}\langle \xi,e_n\rangle
e_n,\ \ \ \xi\in H.
\end{equation}
This gives that $\|e^{tA}\|\leq e^{-t}$ and (A1) holds with $M=1$
and $\nu=-1$. Moreover, due to $\phi$ is Lipschitzian we get from
the H\"older inequality that
\begin{equation*}
\begin{split}
\|b(\xi)-b(\eta)\|_H^2
&\leq L^2\tau\int_{-\tau}^0\int_0^\pi
(\xi(t+\theta,x)-\eta(t+\theta,x))^2 \d x\d\theta\\
&\leq L^2\tau^2\|\xi-\eta\|_H^2,\ \ \ \xi,\eta\in H.
\end{split}
\end{equation*}
Hence (A2) holds with $\rho_1=L\tau$. By the definition of
Hilbert-Schimdt, together with \eqref{eq02}, it follows that
\begin{equation*}
\int_0^Tt^{-2\beta}\|e^{tA}\|_{HS}^2\d
s=\int_0^Tt^{-2\beta}\sum\limits_{n=1}^\infty\|e^{tA}e_n\|_H^2\d
s=\int_0^Tt^{-2\beta}\sum\limits_{n=1}^\infty e^{-2n^2t}\d s.
\end{equation*}
Thus $(A3)$ holds for any $\beta\in(0,\frac{1}{4})$ and
$\rho_2=0$. Furthermore note that
\begin{equation}\label{eq03}
(-A)^{-\frac{1}{2}}\xi=\sum\limits_{n=1}^\infty\frac{1}{n}\langle\xi,e_n\rangle
e_n,\ \ \ \xi\in H,\ \ \
(-A)^{\frac{1}{2}}\xi=\sum\limits_{n=1}^\infty
n\langle\xi,e_n\rangle e_n,\ \ \ \xi\in\mathcal
{D}((-A)^{\frac{1}{2}})
\end{equation}
which in particular yields that $\|(-A)^{-\frac{1}{2}}\|=1$. As a
result, recalling that
$\varphi(\cdot,\cdot,0)=\varphi(\cdot,\cdot,\pi)=0$, we derive  from
\eqref{eq04}, \eqref{eq03} and the H\"older inequality that
\begin{equation*}
\begin{split}
\|(-A)^{\frac{1}{2}}(G(X_t)-G(Y_t))\|_H^2&=\Big\|\sum\limits_{n=1}^\infty
n\langle G(X_t)-G(Y_t),e_n\rangle_He_n\Big\|_H^2\\
&=\sum\limits_{n=1}^\infty\Big(
n\int_0^\pi(G(X_t)(x)-G(Y_t)(x))e_n(x)\d x\Big)^2\\
&=\sum\limits_{n=1}^\infty\Big(
n\int_0^\pi\int_0^\pi\int_{-\tau}^0\varphi(\theta,\xi,x)Z(t+\theta,\xi)\d\theta\d\xi e_n(x)\d x\Big)^2\\
&=\sum\limits_{n=1}^\infty\Big(
\int_0^\pi\int_0^\pi\int_{-\tau}^0\frac{\partial}{\partial
x}\varphi(\theta,\xi,x)Z(t+\theta,\xi)\d\theta\d\xi \tilde{e}_n\d
x\Big)^2\\
&=\sum\limits_{n=1}^\infty\Big\langle\int_0^\pi\int_{-\tau}^0\frac{\partial}{\partial
x}\varphi(\theta,\xi,\cdot)Z(t+\theta,\xi)\d\theta\d\xi,\tilde{e}_n\Big\rangle_H^2\\
&=\Big\|\int_0^\pi\int_{-\tau}^0\frac{\partial}{\partial
x}\varphi(\theta,\xi,\cdot)Z(t+\theta,\xi)\d\theta\d\xi\Big\|_H^2\\
&=\int_0^\pi\Big(\int_0^\pi\int_{-\tau}^0\frac{\partial}{\partial
x}\varphi(\theta,\xi,x)Z(t+\theta,\xi)\d\theta\d\xi\Big)^2\d x\\
&\leq\int_0^\pi\Big\{\int_0^\pi\int_{-\tau}^0\Big(\frac{\partial}{\partial
x}\varphi(\theta,\xi,x)\Big)^2\d\theta\d\xi\\
&\quad\times\int_0^\pi\int_{-\tau}^0Z^2(t+\theta,\xi)\d\theta\d\xi\Big\}\d
x\\
&\leq N\tau\|X_t-Y_t\|^2_\infty,
\end{split}
\end{equation*}
where  $Z(t):=X(t)-Y(t)$ and $\tilde{e}_n:=\sqrt{\frac{2}{\pi}}\cos
nx$, which is also a complete orthonormal system of $H$.
Consequently the law of Eq. \eqref{eq01} $\mathbb{P}_\psi\in T_2(C)$
for some $C>0$ under the metric $\d_{\infty}$ whenever $N\tau$ is
sufficiently small. }
\end{exa}

\begin{rem}
{\rm In this paper we discuss TCIs for neutral functional SDEs and
SPDEs driven by Brownian motion, where the Girsanov transformation
plays an important role as we have explained. As pointed in
\cite[Remark 2.3]{W10}, this approach is, however, unavailable for
the jump-process cases. In the future, we shall establish
by the  Malliavin calculus method on the Poisson space the $W_1H$
transportation inequalities for the distributions of the {\it
segment processes} for a class of neutral functional SDEs with
jumps. }
\end{rem}

\begin{rem}
{\rm There are many interesting applications of the TCIs, e.g., in
Tsirel'son-type inequality and Hoeffding-type inequality, see
\cite{DGW04,WZ04,WZ06}, and in concentration of empirical measure
\cite{M10,W10}. }
\end{rem}


\end{document}